\documentclass[11pt,a4paper]{article}
\usepackage{amsfonts,amsgen,amstext,amsbsy,amsopn,amsfonts,amssymb,amscd}
\usepackage[leqno]{amsmath}
\usepackage[amsmath,amsthm,thmmarks]{ntheorem}
\usepackage{epsf,epsfig}
\usepackage{float}
\usepackage{dsfont}
\usepackage{ebezier,eepic}
\usepackage{color}
\usepackage{tikz}
\usepackage{multirow}
\usepackage{mathrsfs}
\usepackage{graphicx}
\usepackage{subfigure}
\newtheorem{thm}{Theorem}[section]

\newtheorem{lem}[thm]{Lemma}

\newtheorem{false statement}{False statement}
\newtheorem{cor}[thm]{Corollary}

\theoremstyle{definition}

\newtheorem{claim}[thm]{Claim}

\newtheorem{conj}[thm]{Conjecture}

\newtheorem{problem}{Problem}

\makeatletter \@addtoreset{equation}{section}

\def\hh{\mathcal{H}}
\def\hm{\mathcal{M}}

\def\hht{\mathcal{T}}
\def\hd{\mathcal{D}}
\def\hhd{\mathscr{D}}
\def\he{\mathcal{E}}
\def\hf{\mathcal{F}}
\def\hg{\mathcal{G}}

\def\ha{\mathcal{A}}
\def\hb{\mathcal{B}}
\def\hp{\mathcal{P}}
\def\hq{\mathcal{Q}}

\def\hc{\mathcal{C}}
\def\hu{\mathcal{U}}
\def\hi{\mathcal{I}}
\def\hj{\mathcal{J}}
\def\hp{\mathcal{P}}

\begin{document}

\title{\bf\Large  Non-trivial intersecting families for direct products\thanks{\scriptsize{E-mail:  $^1$wangjianmath01@scu.edu.cn, $^2$wu@hunnu.edu.cn, $^3$huajunzhang@usx.edu.cn}
}}
\date{}
\author{Jian Wang$^1$, Biao Wu$^2$, Huajun Zhang$^3$\\[10pt]
$^{1}${\small Department of Mathematics, Sichuan University, Chengdu 610065, China}\\[6pt]
$^{2}${\small School of Mathematics and Statistics,
Hunan Normal University,
Changsha 410081, China}\\[6pt]
$^{3}${\small Department of Mathematics,
Shaoxing University,
Shaoxing 312000, China}\\[6pt]
}
\maketitle
\begin{abstract}
A family of sets is intersecting if any two of its sets intersect, and  non-trivial intersecting if there is not any element which appears in every set of the family.
The study of intersection problems (or Erd\H{o}s-Ko-Rado problems) on families of sets is one of the most important topics in extremal combinatorics.
A natural extension of the Erd\H{o}s-Ko-Rado problem in direct products setting was introduced by Frankl, and  developed by  Alon and Katona et al.
In 2018, Kwan, Sudakov and Vieira [J. Combin. Theory, Ser. A 156 (2018), 44--60] determine the maximum size of a non-trivial intersecting direct product family under the assumption that the sizes of all parts are sufficiently large relative to the set sizes in all parts.
Meanwhile, they raised the question of whether this result holds for the size of each part which depends only on its corresponding set size, and the question of characterizing the extremal families.
In the present paper, we answere these two questions.
Furthermore, the bounds on the size of each part in our result are best possible.
\end{abstract}

\section{Introduction}

Let $[n]$ denote the standard $n$-element set $\{1,2,\ldots,n\}$. For $1\leq i, j\leq n$, let $[i,j]=\{i,i+1,\ldots,j\}$. If $i>j$, then $[i,j]=\emptyset$. Let $2^{[n]}$ denote the power set of $[n]$ and let $\binom{[n]}{k}$ denote the collection of all $k$-element subsets of $[n]$. A family $\hf\subset \binom{[n]}{k}$ is called {\it intersecting} if $F\cap F'\neq \emptyset$ for all $F,F'\in \hf$. One of the most important results in extremal set theory is the Erd\H{o}s-Ko-Rado Theorem.

\begin{thm}[\cite{EKR}]\label{EKR}
Let $\hf\subset \binom{[n]}{k}$ be an intersecting family with $n\geq 2k$. Then
\begin{align}\label{ineq-ekr}
|\hf|\leq \binom{n-1}{k-1}.
\end{align}
\end{thm}

For $n>2k$, the equality holds if and only if $\hf=\{F\in \binom{[n]}{k}\colon x\in F\}$ for some $x\in [n]$.

Let $\hf\subset \binom{[n]}{k}$. If $\cap_{F\in \hf}F \neq \emptyset$ then $\hf$ is called {\it a star}. If $\hf$ is not a star then we call  $\hf$  {\it non-trivial}. Hilton and Milner \cite{HM} determined the maximum size of a non-trivial intersecting family.

\begin{thm}[\cite{HM}]\label{thm-hm}
 If $\hf\subset \binom{[n]}{k}$ is non-trivial intersecting and $n>2k\geq 4$, then
\begin{align}\label{ineq-nontrival}
|\hf| \leq \binom{n-1}{k-1}- \binom{n-k-1}{k-1} +1.
\end{align}
\end{thm}

Define the {\it Hilton-Milner family} $\hh\hm(n,k)$ as
\[
\hh\hm(n,k) =\left\{F\in \binom{[n]}{k}\colon 1\in F,\ F\cap [2,k+1]\neq \emptyset\right\}\cup \{[2,k+1]\}
\]
and the {\it triangle} family $\hht(n,k)$ as
\[
\hht(n,k) =\left\{F\in \binom{[n]}{k}\colon |F\cap [3]|\geq 2\right\}.
\]
Note that $\hh\hm(n,2)=\hht(n,2)$.
The equality holds in \eqref{ineq-nontrival} if and only if $\hf$ is isomorphic to $\hh\hm(n,k)$ for $k\neq 3$ and isomorphic to either $\hh\hm(n,k)$ or $\hht(n,k)$ for $k=3$.

A natural extension of the Erd\H{o}s-Ko-Rado problem in direct product setting was introduced by Frankl~\cite{F1996}, in connection with a result of Sali~\cite{Sal92} (see also~\cite{FT93}).

Let $n_1,n_2,\ldots,n_p,k_1,k_2,\ldots,k_p$ be positive integers.
For $\hf_i\subset 2^{[n_i]}$, $i=1,2,\ldots,p$,
the {\it direct product} of $\hf_1,\hf_2,\ldots,\hf_p$ is defined as
\[
\prod_{1\leq i\leq p} \hf_i =\left\{(F_1,F_2,\ldots,F_p)\colon F_i\in\hf_i,\ i=1,2,\ldots,p\right\}.
\]
For simplicity, we often use $\vec{F}$ to denote $(F_1,F_2,\ldots,F_p)$, and similarly for $\vec{E}$ and $\vec{A}$.
For $\vec{F}\in \prod\limits_{1\leq i\leq p}2^{[n_i]}$, define
\[
|\vec{F}| = |F_1|+|F_2|+\cdots+|F_p|.
\]
For $\vec{F},\vec{F}'\in \prod\limits_{1\leq i\leq p}2^{[n_i]}$, define
\[
\vec{F}\cap \vec{F}' =(F_1\cap F_1',F_2\cap F_2',\ldots, F_p\cap F_p')\mbox{ and } \vec{F}\cup \vec{F}' =(F_1\cup F_1',F_2\cup F_2',\ldots, F_p\cup F_p').
\]
A family $\hf\subset \prod\limits_{1\leq i\leq p}2^{[n_i]}$ is called a {\it direct product} family. We call $\hf\subset \prod\limits_{1\leq i\leq p}2^{[n_i]}$  {\it intersecting} if $|\vec{F}\cap \vec{F}'|\geq 1$ for all $\vec{F}, \vec{F}'\in \hf$.

By using the algebraic method, Frankl \cite{F1996}  proved an Erd\H{o}s-Ko-Rado theorem for direct product families.

\begin{thm}[\hspace{-0.001em}\cite{F1996}]\label{EKR-directprod}
Let $n_1,n_2,\ldots,n_p,k_1,k_2,\ldots,k_p$ be positive integers with  $2\leq n_1/k_1\leq \dots \leq n_p/k_p$. Suppose that $\hf\subseteq \prod\limits_{1\leq i\leq p} \binom{[n_i]}{k_i}$ is intersecting. Then
\[
|\hf| \leq \frac{k_1}{n_1}\prod_{i=1}^p\binom{n_i}{k_i}.
\]
\end{thm}

We note that Frankl’s theorem may be viewed as a statement concerning the size of the maximum independent set in a particular product graph. Research on independent sets in graph products--especially in graph powers--has a rich tradition, with a substantial body of literature dedicated to this area.
In 2016, Frankl's result (Theorem \ref{EKR-directprod}) was rediscovered by Alon and Katona \cite{AlonKa}.
Let ${\hf}\subset \prod\limits_{1\leq i\leq p}\binom{[n_i]}{k_i}$. We say that $\hf$ is {\it a star of the $i$th coordinate} with centre $x$ if  $x\in F_i$ for all $\vec F\in {\hf}$. A family ${\hf}\subset \prod\limits_{1\leq i\leq p}\binom{[n_i]}{k_i}$ is called {\it non-trivial} if it is not a star of the $i$th coordinate for all $i=1,2,\ldots,p$.
Alon and Katona \cite{AlonKa} also made the following conjecture.

\begin{conj}[\cite{AlonKa}]\label{conj-1}
Let $n_1\geq 2k_1$ and $n_2\geq 2k_2$.
Suppose that $\hf\subseteq  \binom{[n_1]}{k_1}\times \binom{n_2]}{k_2}$ is non-trivial intersecting. Then
\[
|\hf| \leq \max\left\{|\hh\hm(n_1,k_1)|\binom{n_2}{k_2},\binom{n_1}{k_1}|\hh\hm(n_2,k_2)|\right\}.
\]
\end{conj}

In 2018, Kwan, Sudakov and Vieira \cite{KSV} disproved Conjecture \ref{conj-1}. In general, they determined the maximum size of a non-trivial intersecting direct product family for large $n_i$, $i\in [p]$.
 Before stating their result, let us introduce some constructions. Let $\vec{n}=(n_1,n_2,\dots,n_p)$ and $\vec{k}=(k_1,k_2,\dots,k_p)$.
 For any $t\in [p]$ and $S\subset [p]\setminus \{t\}$, define the families $\hh\hm_1(\vec{\mathstrut n},\vec{\mathstrut k},S,t)$ and $\hh\hm_2(\vec{\mathstrut n},\vec{\mathstrut k},S,t)$ as follows.
If $\sum_{j\in S} k_j+k_t \geq 2$,  the family $\hh\hm_1(\vec{\mathstrut n},\vec{\mathstrut k},S,t)$ consists of all sets $\prod\limits_{1\leq i\leq p}\binom{[n_i]}{k_i}$ such that one of the following conditions is satisfied:
 \begin{itemize}
  \item $F_t=[2,k_t+1]$ and $F_s=[k_s]$ for all $s\in S$, or,
  \item $1\in F_t$ and either $F_t \cap [2,k_t+1]\neq \emptyset $ or $F_s\cap [k_s]\neq \emptyset$ for some $s\in S$.
\end{itemize}
 If $\sum_{j\in S} k_j \geq 2$,  the family $\hh\hm_2(\vec{\mathstrut n},\vec{\mathstrut k},S,t)$ consists of all sets $\prod\limits_{1\leq i\leq p}\binom{[n_i]}{k_i}$ such that one of the following conditions is satisfied:
 \begin{itemize}
  \item $F_s=[k_s]$ for all $s\in S$, or,
  \item $1\in F_t$ and  $F_s\cap [k_s]\neq \emptyset$ for some $s\in S$.
\end{itemize}
The families $\hh\hm_1(\vec{\mathstrut n},\vec{\mathstrut k},S,t)$ and $\hh\hm_2(\vec{\mathstrut n},\vec{\mathstrut k},S,t)$ are non-trivially intersecting with
 \begin{small}\[
|\hh\hm_1(\vec{\mathstrut n},\vec{\mathstrut k},S,t)|=\left(\binom{n_{t}-1}{k_{t}-1}\prod_{s\in S}\binom{n_{s}}{k_{s}}-\binom{n_{t}-k_{t}-1}{k_{t}-1}\prod_{s\in S}\binom{n_{s}-k_{s}}{k_{s}}+1\right)\prod_{s\notin S\cup\left\{ t\right\} }\binom{n_{s}}{k_{s}},
\]\end{small}
and
 \begin{small}\[
|\hh\hm_2(\vec{\mathstrut n},\vec{\mathstrut k},S,t)|=\left(\binom{n_{t}-1}{k_{t}-1}\left(\prod_{s\in S}\binom{n_{s}}{k_{s}}-\prod_{s\in S}\binom{n_{s}-k_s}{k_{s}}\right)+\binom{n_{t}-1}{k_{t}}\right)\prod_{s\notin S\cup\left\{ t\right\} }\binom{n_{s}}{k_{s}}.
\]\end{small}
Let
 \begin{small}
 \[
 \Omega_1(\vec{\mathstrut n},\vec{\mathstrut k}):=\left\{(S,t)\colon  t\in [p],\ S\subset [p]\setminus \{t\}, \sum_{j\in S} k_j+k_t \geq 2\right\}
 \]
 and
 \[
 \Omega_2(\vec{\mathstrut n},\vec{\mathstrut k}):=\left\{(S,t)\colon  t\in [p],\ S\subset [p]\setminus \{t\},  \sum_{j\in S} k_j \geq 2\right\}.
 \]\end{small}
 Define
\begin{align*}
M_1(\vec{\mathstrut n},\vec{\mathstrut k}):= \max_{(S,t)\in \Omega_1(\vec{n},\vec{k})} |\hh\hm_1(\vec{\mathstrut n},\vec{\mathstrut k},S,t)| \:\: {\rm and}\:\:
M_2(\vec{\mathstrut n},\vec{\mathstrut k}):= \max_{(S,t)\in \Omega_2(\vec{n},\vec{k})} |\hh\hm_2(\vec{\mathstrut n},\vec{\mathstrut k},S,t)|.
\end{align*}

\begin{thm}[\cite{KSV}]\label{thm-0}
For any $k_1,\ldots,k_p\geq 1$, there is an integer $n_0=n_0(k_1,\ldots,k_p)$ such that if $n_1,\ldots,n_p\geq n_0$ and if ${\hf} \subset \prod\limits_{1\leq i\leq p} \binom{[n_i]}{k_i}$ is a non-trivial intersecting family, then
\[
|{\hf}| \leq \max \{M_1(\vec{\mathstrut n},\vec{\mathstrut k}),M_2(\vec{\mathstrut n},\vec{\mathstrut k})\}.
\]
\end{thm}

In the historical development of the Erd\H{o}s-Ko-Rado problem, determining the threshold of the parameter $n$ for which a certain extremal phenomenon persists is a very important problem. For instance, the lower bound  $n \ge (t+1)(k-t+1)$ in the Erd\H{o}s-Ko-Rado theorem was completely resolved step by step by several groups of scholars.
At the same time, determining the structure of extremal set families is also a very natural, fundamental, and important problem.
Kwan, Sudakov and Vieira \cite{KSV} proposed the following two problems.

\begin{problem}[\cite{KSV}]
Whether the bound in Theorem \ref{thm-0} remains valid when we do not make the assumption that the parts are large. As an intermediate problem, one might hope to prove a version of Theorem \ref{thm-0} where each $n_s$ is only assumed to be large relative to its corresponding $k_s$, not necessarily relative to all the $k_s$ at once.
\end{problem}

\begin{problem}[\cite{KSV}]
How to deal with potential non-shifted extremal families in Theorem \ref{thm-0}.
\end{problem}

For some $t\in [p]$ with $k_t\geq 3$, define the {\it generalized triangle family} $\hht(\vec{\mathstrut n},\vec{\mathstrut k},t)$ as
\begin{align*}
\hht(\vec{\mathstrut n},\vec{\mathstrut k},t) =&\left\{\vec{F}\in\prod_{1\leq i\leq p} \binom{[n_i]}{k_i}\colon |F_t\cap [3]|\geq 2\right\}.
\end{align*}

In the present paper,  we answer Problem 1 and  Problem 2.
Furthermore, the bounds on each $n_i$ in our result are best possible.

\begin{thm}\label{thm-1}
Let $n_1,\ldots, n_p, k_1,\ldots,k_p\geq 1$ with $n_i>2k_i$, $i=1,2,\ldots,p$.
If ${\hf} \subset \prod\limits_{1\leq i\leq p} \binom{[n_i]}{k_i}$ is a non-trivial intersecting family, then
\[
|{\hf}| \leq \max \{M_1(\vec{\mathstrut n},\vec{\mathstrut k}),M_2(\vec{\mathstrut n},\vec{\mathstrut k})\}.
\]
Moreover, the equality holds if and only if  $\hf$ is isomorphic to either $\hh\hm_1(\vec{\mathstrut n},\vec{\mathstrut k},S,t)$ for some $(S,t)\in  \Omega_1(\vec{\mathstrut n},\vec{\mathstrut k})$ or $\hh\hm_2(\vec{\mathstrut n},\vec{\mathstrut k},S,t)$ for some $(S,t)\in  \Omega_2(\vec{\mathstrut n},\vec{\mathstrut k})$ or $\hht(\vec{\mathstrut n},\vec{\mathstrut k},t)$ for some $t\in [p]$ with $k_t=3$.
\end{thm}

We note that recently Xi, Kong, and Ge \cite{XKG} considered the extremal problem of the sum of cardinalities of cross‑intersecting families for direct products.

{\bf Brief proof outline. }
Suppose that ${\hf} \subset \prod\limits\binom{[n_i]}{k_i}$ is a non-trivial intersecting family of maximum size.
We first show that there is a non-trivial intersecting family ${\hf} $ which is shifted.
Secondly, we connect this ${\hf}$ with its ``core'', which is called the generating set $g(\hf)$.
The generating set is a compact object that contains almost all the information about ${\hf}$ and has many critical properties.
We claim that $g(\hf)$ has a very good structure.
A proof by contradiction will be used to establish this result.
If its structure were otherwise, one could invariably construct a larger structure. This contradicts the assumption that $\hf$ is the maximal non‑trivial intersecting family, forcing it to be a star. A brief analysis of this scenario then yields a contradiction.
Lastly, based on the structure of the generating set, we provide a complete proof of the main theorem.

The rest of the paper will be organized as follows. In Section 2, we recall the shifting method and the generating set method, and adapt them into the direct product setting. In Section 3, we obtain the structures of shifted maximum intersecting families. In Section 4, we prove Theorem \ref{thm-1}.

\section{The shifting method and the generating set}

In this section, we recall two powerful tools in extremal set theory. One is the  shifting method and the other one is the generating set method.

The  shifting method was invented by Erd\H{o}s, Ko and Rado \cite{EKR}. We shall use a direct product adaptation of the shifting operation, which was given in \cite{KSV}. For ${\hf} \subset \prod\limits_{1\leq i\leq p} \binom{[n_i]}{k_i}$, $t\in [p]$ and $1\leq i<j\leq n_t$, define the  $(i,j)$-shift of the $t$th coordinate as
$$S_{i,j}^t(\hf)=\left\{S_{i,j}^t(\vec{F})\colon \vec{F}\in{\hf}\right\},$$
where
$$S_{i,j}^t(\vec{F})=\left\{
                \begin{array}{ll}
                 \vec{F}':=(F_1,\ldots,(F_t\setminus\{j\})\cup \{i\},\ldots,F_p), & \text { if }j\in F_t, i\notin F_t \text{ and } \vec{F}'\notin {\hf}; \\[5pt]
                  \vec{F}, & \hbox{otherwise.}
                \end{array}
              \right.
$$
 It is well-known (cf. \cite{F87}) that $S_{i,j}^t$ preserves both the size $|{\hf}|$ and the intersecting property.

 For ${\hf}\subset \prod\limits_{1\leq i\leq p} \binom{[n_i]}{k_i}$, we call ${\hf}$ a {\it shifted family} if $S_{i,j}^t({\hf})={\hf}$ for all $1\leq t\leq p$ and all $i,j$ with $1\leq i<j\leq n_t$.

An intersecting family $\hf\subset \prod\limits_{1\leq i\leq p}\binom{[n_i]}{k_i}$ is called {\it maximum} if any addition of an
 extra set of $\prod\limits_{1\leq i\leq p}\binom{[n_i]}{k_i}$ to $\hf$ will destroy the intersecting property. Define a partial order $\preceq$ on $\prod\limits_{1\leq i\leq p} 2^{[n_i]}$ as follows. For $\vec{E}, \vec{E}' \in \prod\limits_{1\leq i\leq p} 2^{[n_i]}$,  define $\vec{E}\preceq \vec{E}'$ if  $E_i\subset E_i'$ for all $i=1,2,\ldots,p$. If $\vec{E}\preceq \vec{E}'$ and $\vec{E}\neq  \vec{E}'$ then we write $\vec{E}\prec \vec{E}'$.

 We need the following lemma, which is contained in \cite{KSV}. 

\begin{lem}\label{thm-key}
 Let $p\geq 1$, $n_1,\ldots,n_p, k_1,\ldots,k_p\geq 1$ be integers. If ${\hf}\subset \prod\limits_{1\leq i\leq p} \binom{[n_i]}{k_i}$ is a non-trivial intersecting family of the maximum size, then
 one can arrive at a shifted non-trivial intersecting family ${\hf}'\subset \prod\limits_{1\leq i\leq p} \binom{[n_i]}{k_i}$ by applying the shifting operation $S_{i,j}^t$ to $\hf$ for all $1\leq t\leq p$ and all $1\leq i<j\leq n_t$ repeatedly.
\end{lem}

We now recall another powerful tool, known as the generating set method (or the AK-method), introduced by Ahlswede and  Khachatrian \cite{AK1996,AK}. Here we adapt it to the direct product setting.


For $\vec{E}=(E_1,E_2,\ldots, E_p)\in  \prod\limits_{1\leq i\leq p} 2^{[n_i]}$, define the {\it up-set} $\hu(\vec{E})$ as
\[
\hu(\vec{E})=\left\{\vec F \in \prod\limits_{1\leq i\leq p} 2^{[n_i]}\colon E_i\subset F_i\subset [n_i], \ i=1,2,\ldots,p\right\}.
\]
For ${\he}\subset \prod\limits_{1\leq i\leq p} 2^{[n_i]}$, define
$
\hu(\he)=\bigcup_{\vec{E}\in {\he}} \hu(\vec{E}).
$
Let ${\hf}\subset \prod\limits_{1\leq i\leq p} \binom{[n_i]}{k_i}$. A family $\hg\subset \prod\limits_{1\leq i\leq p} 2^{[n_i]}$ is called a  {\it generating set} of ${\hf}$ if
\[
\hu(\hg) \bigcap \prod\limits_{1\leq i\leq p} \binom{[n_i]}{k_i} ={\hf}.
\]

 It is clear that $\mathcal{F}$ is a generating set of itself. The set of all generating sets of $\mathcal{F}$ forms a nonempty set which we denote by $G(\mathcal{F})$.
 For $g(\mathcal F)\in G(\mathcal F)$, let $g_*(\mathcal F)$ be the set of all minimal (in the sense of
set-theoretical inclusion) elements of $g(\mathcal F)$. Set \[G_*(\mathcal F)=\{g(\mathcal F)\in G(\mathcal F): g(\mathcal F)=g_*(\mathcal F)\}.\]

For $\ha\subset g(\hf)$ and $1\leq i\leq p$, define {\it the projection} of $\ha$ at the $i$th coordinate as
\[
\pi_i(\ha) = \left\{A_i\colon \vec A\in \ha\right\}.
\]
For $E\subset [n]$, let $s^+(E)=\max E$ and set $s^+(\emptyset)=0$. For $\he\subset 2^{[n]}$, let $s^+(\he) = \max\{s^+(E)\colon E\in \he\}$. Let $\ell_i= s^+(\pi_i(g(\hf)))$, $i=1,2,\ldots,p$ and define $s^+(g(\hf)):=(\ell_1,\ell_2,\ldots,\ell_p)$. Set
\[
g_i^{*}(\hf)=\left\{\vec E\in g(\hf)\colon \ell_i\in E_i\right\}
\: {\rm and}\:
g_{i,a_i}^{*}(\hf)=\left\{\vec E \in g(\hf)\colon \ell_i\in E_i,\ |E_i|=a_i\right\}.
\]
For $\vec E\in \he \subset g_{i}^{*}(\hf)$, denote
\[\vec E_{(i)}'=(E_1,\cdots,E_{i-1},E_i\setminus \{\ell_i\},E_{i+1},\cdots,E_p) \: {\rm and } \:
\he_{(i)}' = \{ \vec E_{(i)}': \vec E\in \he\}.\]
For $\vec{E}\in g(\hf)$ and $\vec E_{(i)}'\in \he_{(i)}'$, define
 \[
  \hhd(\vec{E}) := \left\{\vec{F}\in \prod_{1\leq i\leq p} \binom{[n_i]}{k_i}\colon F_i\cap [s^+(E_i)] =E_i,\ i=1,2,\ldots,p\right\}
  \]
  and
  \[
  \hd(\vec E_{(i)}') := \left\{\vec{F}\in \prod_{1\leq i\leq p} \binom{[n_i]}{k_i}\colon F_i\cap [s^+(E_i)] =E_i \setminus \{\ell_i\}, F_j\cap [s^+(E_j)] =E_j, j\neq i \right\},
  \]
  where $[0]=\emptyset$.
 Clearly,
 \[
 |\hhd(\vec{E})| = \prod_{1\leq i\leq p} \binom{n_i-s^+(E_i)}{k_i-|E_i|}.
 \]
 For $\vec{E}\in g_{i,a_i}^{*}(\hf)$, we have
 \[
 |\hd(\vec E_{(i)}'| = \binom{n_i-\ell_i}{k_i-a_i+1}\prod_{j\in [p]\setminus \{i\}} \binom{n_j-s^+(E_j)}{k_j-|E_j|}=\frac{n_i-\ell_i-k_i+a_i}{k_i-a_i+1} |\hhd(\vec{E})|.
 \]
Denote
\[
\hhd(\ha) =\bigcup_{\vec{A}\in \ha} \hhd(\vec{A}) \:{\rm and} \: \hd(\ha) =\bigcup_{\vec{A}\in \ha} \hd(\vec{A}).
\]

To gain some intuitive understanding of these definitions, let's consider a simple example of \[ \hf=\left\{\vec F \in \binom{[n_1]}{2}\times \binom{[n_2]}{3}: 1\in F_1 \:{\rm and}\: F_2\cap \{1,2,3\} \neq \emptyset\right\}\cup \{(\emptyset,\{1,2,3\} )\}.\]

\begin{itemize}
  \item   $g(\hf)=\{(\{1\},\{1\}),\:(\{1\},\{2\}),\:(\{1\},\{3\}),\:(\emptyset,\{1,2,3\})\}$;
  \item $\ell_1=1$, $\ell_2=3$;
  \item  $\pi_1(g(\hf))=\{\{1\},\emptyset\}$, \:$\pi_2(g(\hf))=\{\{1\},\{2\}, \{3\},\{1,2,3\}\}$;
   \item  $g_{2}^{*}(\hf)=\{(\{1\},\{3\}),(\emptyset,\{1,2,3\})\}$, \:$g_{2,3}^{*}(\hf)=\{(\emptyset,\{1,2,3\})\}$;
    \item  $(\emptyset,\{1,2,3\})_{(2)}'=(\emptyset,\{1,2\})$;
    \item      $\hd((\emptyset,\{1,2,3\})_{(2)}')=\{ \vec F\in \binom{[n_1]}{2}\times \binom{[n_2]}{3}:  F_2\cap \{1,2,3\}=\{1,2\}\}$;
    \item  $\hhd((\emptyset,\{1,2,3\}))=\{ \vec F\in \binom{[n_1]}{2}\times \binom{[n_2]}{3}:  \{1,2,3\} \subset F_2\}$.
  \end{itemize}

It is clear that $g(\hf)$ is intersecting if $\hf $ is intersecting. Let us show  that the generating set has the following properties.

\begin{lem}\label{lem-key}
Suppose that $\hf\subset \prod\limits_{1\leq i\leq p}\binom{[n_i]}{k_i}$ is a non-trivial intersecting family of the maximum size and $n_i> 2k_i$, $i=1,2,\ldots,p$.  Let $g(\hf)\in G_*(\mathcal F)$ with  $s^+(g(\hf))=(\ell_1,\ell_2,\ldots,\ell_p)$ such that $\sum_{\vec F\in g(\hf)}|\vec F|$ is minimal. Then
\begin{itemize}
  \item[{\rm  (i)}] Suppose that $t\in [p]$, $1\leq i<j\leq n_t$ and $\vec{E}=(E_1,\ldots,E_t,\ldots,E_p)\in g(\hf)$ with $E_t\cap \{i,j\}=\{j\}$. Let $\vec{E}'=(E_1,\ldots,(E_t\setminus \{j\})\cup \{i\},\ldots,E_p)$. Then there exists $\vec{E}''\in g(\hf)$ such that $\vec{E}''\preceq \vec{E}'$.
  \item[{\rm (ii)}]  $\hf $ is a disjoint union $\bigcup_{\vec{A}\in g(\hf)} \hhd(\vec{A})=\hhd(g(\hf))$.
  \item[{\rm (iii)}]  For any $\vec{A}\in g_{i,a_i}^{*}(\hf)$, there  exists $\vec{B}\in g_{i,\ell_i+1-a_i}^{*}(\hf)$ such that $\vec{A}\cap \vec{B}=(\emptyset,\ldots,\emptyset, \{\ell_i\},$\\$\emptyset,\ldots,\emptyset)$. Moreover, if $\vec{A}\cap \vec{B}=(\emptyset,\ldots,\emptyset, \{\ell_i\},\emptyset,\ldots,\emptyset)$ then $A_i\cup B_i=[\ell_i]$.
\end{itemize}
\end{lem}

\begin{proof}
First we prove (i). Since $\hhd(\vec{E})\subset \hf$ and $\hf$ is shifted, we infer that $\hhd(\vec{E}')\subset \hf$.
Combining the definition of $g(\hf)$,  it follows that either $\vec{E}'\in g(\hf)$ or there exists $\vec{E}''\in g(\hf)$ such that $\vec{E}''\prec \vec{E}'$. Thus, (i) holds.

Next we show (ii). Since $\hf$ is miximum, we infer that $\hhd(\vec{A})\subset \hf$ for all $\vec{A}\in g(\hf)$. It follows that $\hhd(g(\hf))\subset \hf$. On the other direction, choose an arbitrary $\vec{F}=(F_1,F_2,\ldots,F_p)\in \hf$.
Then there exists $\vec{A}\in g(\hf)$ such that $\vec F \in \hu(\vec A)$, which implies that $F_i\cap [s^+(A_i)] =A_i$ holds for each $i\in [p]$.
Therefore, $\vec F \in \hhd(\vec A)$ and hence $\hf \subset \hhd(g(\hf))$.
It follows that $\hf = \hhd(g(\hf))$.
For any different $\vec F, \vec G \in g(\hf)$,  $\hhd(g(\hf)) \cap \hhd(g(\hg))=\emptyset$ since $g(\hf)\in G_*(\mathcal F)$, that is, $g(\hf)$ is an anti-chain.
Thus $\hf $ is a disjoint union $\bigcup_{\vec{A}\in g(\hf)} \hhd(\vec{A})$.

We are left to show (iii). Let $\vec{A}=(A_1,\ldots,A_i,\ldots,A_p)\in g_{i,a_i}^{*}(\hf)$ and let $\vec{A}'=(A_1,\ldots,A_i\setminus \{\ell_i\},\ldots,A_p)$.
Suppose that $\vec A' \cap \vec C\neq \emptyset$ for all $\vec C \in g(\hf)$, then $\hhd(g')$ is also intersecting with $|\hhd(g')|\ge |\hf|$, where $g'=g(\hf)\cup \{\vec A'\}\setminus \{\vec A\}$.
This contradicts that $\sum_{\vec F\in g(\hf)}|\vec F|$ is minimal.
Hence, there exists some $\vec B\in g(\hf)$ such that $\vec{A}\cap \vec{B}=(\emptyset, \ldots,\emptyset,\{\ell_i\},\emptyset,\ldots,\emptyset)$, and $\vec{B}\in g_i^*(\hf)$ follows.
Suppose that  $A_i\cup B_i\subsetneq [\ell_i]$. Let $x\in [\ell_i]\setminus (A_i\cup B_i)$. By (i) there exists $\vec{B}'\in g(\hf)$ such that $\vec{B}'\preceq (B_1,\ldots,(B_i\setminus \{\ell_i\})\cup \{x\},\ldots,B_p)$.
It follows that $|\vec{A}\cap \vec{B}'|=0$, contradicting the fact that $g(\hf)$ is intersecting. Thus $A_i\cup B_i=[\ell_i]$ and $\vec{B}\in g_{i,\ell_i+1-a_i}^*(\hf)$. Therefore (iii) holds.
\end{proof}

{\noindent\bf Remark.} Note that for any $i$ there exists some $a_i$ such that $g^*_{i,a_i}(\hf)\neq \emptyset$.  By Lemma \ref{lem-key} (iii) if $g^*_{i,a_i}(\hf)\neq \emptyset$ then $g^*_{i,\ell_i+1-a_i}(\hf)\neq \emptyset$.
Since $a_i\leq k_i$ and $\ell_i+1-a_i\leq k_i$,  we infer $\ell_i<2k_i$ for all $i=1,2,\ldots,p$.

The following fact shows that the generating set of a direct product family has  different properties from that of an ordinary family.

\begin{lem}\label{lem-key4}
Let $n_i> 2k_i$, $i=1,2,\ldots,p$. Suppose that $\hf\subset \prod\limits_{1\leq i\leq p}\binom{[n_i]}{k_i}$ is a non-trivial intersecting family of the maximum size.
Let $g(\hf)\in G_*(\mathcal F)$ with  $s^+(g(\hf))=(\ell_1,\ell_2,\ldots,\ell_p)$ such that $\sum_{\vec F\in g(\hf)}|\vec F|$ is minimal.
Let $a_i,b_i\leq k_i$ be positive integers with $a_i+b_i=\ell_i+1$. Let  $\ha\subset g^*_{i,a_i}(\hf)$, $\hb\subset g^*_{i,b_i}(\hf)$, $\hf_1=\hhd(g(\hf)\setminus \hb)\cup \hd(\ha_{(i)}')$ and  $\hf_2=\hhd(g(\hf)\setminus \ha)\cup \hd(\hb_{(i)}')$. If $\ha$ and $ \hb$ satisfy the following three conditions,
\begin{itemize}
  \item[{\rm(C1)}] $\ha\neq \emptyset$, $\hb\neq \emptyset$ and $\ha \cap \hb =\emptyset$,
  \item[{\rm(C2)}] for $\vec A\in g^*_{i,a_i}(\hf)$, if $\vec A\cap \vec B=(\emptyset,\cdots,\emptyset,\{\ell_i\},\emptyset,\cdots,\emptyset)$  for some $\vec B \in \hb $ then $\vec A\in \ha$,
  \item[{\rm(C3)}] for $\vec B\in g^*_{i,b_i}(\hf)$, if $\vec B\cap \vec A=(\emptyset,\cdots,\emptyset,\{\ell_i\},\emptyset,\cdots,\emptyset)$  for some $\vec A \in \ha $ then $\vec B\in \hb$,
 \end{itemize}
 then both $\hf_1$ and $\hf_2$ are intersecting with $\max\{|\hf_1|,|\hf_2|\}>|\hf|$.
 In particular, if $\hf$ is non-trivial intersecting of maximal size, then one of $\hf_1$ and $\hf_2$ is a star.

\end{lem}
\begin{proof}
By condition (C3),
we infer that $\vec C\cap \vec F\neq \emptyset$ for all $\vec C\in \ha_{(i)}'$ and $\vec F\in g(\hf)\setminus \hb$.
Consequently, $g(\hf)\cup \ha_{(i)}'\setminus \hb$ is intersecting, and hence $\hf_1=\hhd(g(\hf)\setminus \hb)\cup \hd(\ha_{(i)}')$  is intersecting.
Similarly, $\hf_2$ is intersecting.

For  $\vec E\in \he \subset g(\hf)$, let $w_{i}(\vec E) \binom{n_i-s^+(E_i)}{k_i-a_i} $ be the number of sets in $\hf$ such that $F_j\cap [s^+(E_j)]=E_j$ for all $j\in[p] $, that is,
\[w_{i}(\vec E)=|\{(F_1,\cdots,F_{i-1},E_i,F_{i+1},\cdots,F_p): \exists \vec{F}\in \hf  \:{\rm such \: that} \: F_j\cap [s^+(E_j)]=E_j, j\in[p] \}|.\]
Denote $w_{i}(\he)=\sum_{\vec E\in \he}w_i(\vec E)$.
One can derive that \[|\hhd(\ha)|=w_{i}(\ha) \binom{n_i-\ell_i}{k_i-a_i}\: {\rm and} \:|\hd(\ha_{(i)}')\setminus \hf|=  w_{i}(\ha)\binom{n_i-\ell_i}{k_i-a_i+1}.\]
Similarly, \[|\hhd(\hb)|=w_{i}(\hb) \binom{n_i-\ell_i}{k_i-b_i}\: {\rm and} \:|\hd(\hb_{(i)}')\setminus \hf|=  w_{i}(\hb)\binom{n_i-\ell_i}{k_i-b_i+1}.\]
Thus
\begin{align}\label{ineq-2.5}
|\hf_1|=  |\hf| -w_{i}(\hb)\binom{n_i-\ell_i}{k_i-b_i} +w_{i}(\ha)\binom{n_i-\ell_i}{k_i-a_i+1},
\end{align}
and
\begin{align}\label{ineq-2.6}
|\hf_2|=  |\hf| -w_{i}(\ha)\binom{n_i-\ell_i}{k_i-a_i} +w_{i}(\hb)\binom{n_i-\ell_i}{k_i-b_i+1}.
\end{align}
By condition (C1),  we obtain that $w_{i}(\ha)>0$ and $w_{i}(\hb)>0$.
If $\max\{|\hf_1|,|\hf_2|\}\leq |\hf|$, then by \eqref{ineq-2.5} and \eqref{ineq-2.6} we get
\[
 w_{i}(\ha) \binom{n_i-\ell_i}{k_i-a_i+1}\leq  w_{i}(\hb)\binom{n_i-\ell_i}{k_i-b_i},\  w_{i}(\hb) \binom{n_i-\ell_i}{k_i-b_i+1}\leq  w_{i}(\ha) \binom{n_i-\ell_i}{k_i-a_i}.
\]
Multiplying the two inequalities, we get
\[
\binom{n_i-\ell_i}{k_i-a_i+1}\binom{n_i-\ell_i}{k_i-b_i+1}\leq  \binom{n_i-\ell_i}{k_i-a_i}\binom{n_i-\ell_i}{k_i-b_i}.
\]
Equivalently,
\[
\frac{n_i-\ell_i-k_i+a_i}{k_i-a_i+1} \frac{n_i-\ell_i-k_i+b_i}{k_i-b_i+1} \leq 1.
\]
Since $n_i> 2k_i$ and $\ell_i+1=a_i+b_i$,
\[
\frac{n_i-\ell_i-k_i+a_i}{k_i-a_i+1} \frac{n_i-\ell_i-k_i+b_i}{k_i-b_i+1} > \frac{k_i-\ell_i+a_i}{k_i-a_i+1}\frac{k_i-\ell_i+b_i}{k_i-b_i+1}\geq \frac{k_i-b_i+1}{k_i-a_i+1}\frac{k_i-a_i+1}{k_i-b_i+1}=1,
\]
a contradiction. Thus $\max\{|\hf_1|,|\hf_2|\}>|\hf|$.
Consquencely, if $\hf$ is maximal, then one of $\hf_1$ and $\hf_2$ is not non-trivial, that is a star.
\end{proof}

Let $\ha=g^*_{i,a_i}(\hf)$, $\hb= g^*_{i,b_i}(\hf)$ with $b_i=\ell_i+1-a_i\neq a_i$ and $a_i\neq b_i$.
Suppose that $g_{i,a_i}^{*}(\hf)\neq \emptyset$.
Then by Lemma \ref{lem-key}, we know that $\ha$ and $\hb$ satisfy the three conditions in Lemma \ref{lem-key4}.
Thus, we obtain the following corollary immediately.
\begin{cor}\label{lem-key1}
Let $n_i> 2k_i$, $i=1,2,\ldots,p$.
Suppose that $\hf\subset \prod\limits_{1\leq i\leq p}\binom{[n_i]}{k_i}$ is a non-trivial intersecting family of the maximum size.
Let $g(\hf)\in G_*(\mathcal F)$ with  $s^+(g(\hf))=(\ell_1,\ell_2,\ldots,\ell_p)$ such that $\sum_{\vec F\in g(\hf)}|\vec F|$ is minimal.
If there exists $a_i\leq \ell_i$ such that  $g_{i,a_i}^{*}(\hf)\neq \emptyset$ and $b_i:=\ell_i+1-a_i\neq a_i$, then one of
\[
\hhd(g(\hf) \setminus g_{i,b_i}^{*}(\hf)) \cup \hd(g_{i,a_i}^{*}(\hf)_{(i)}')  \mbox{ and } \hhd(g(\hf) \setminus g_{i,a_i}^{*}(\hf)) \cup \hd(g_{i,b_i}^{*}(\hf)_{(i)}')
\]
is a star.
\end{cor}

For $a_i=b_i$, we obtain the following corollary.
\begin{cor}\label{lem-key2}
Let $n_i> 2k_i$, $i=1,2,\ldots,p$.
Suppose that $\hf\subset \prod\limits_{1\leq i\leq p}\binom{[n_i]}{k_i}$ is a non-trivial intersecting family of the maximum size.
Let $g(\hf)\in G_*(\mathcal F)$ with  $s^+(g(\hf))=(\ell_1,\ell_2,\ldots,\ell_p)$ such that $\sum_{\vec F\in g(\hf)}|\vec F|$ is minimal, where $\ell_i\geq 2$. Let $a_i=\frac{\ell_i+1}{2}$.
If  $g_{i,a_i}^{*}(\hf)\neq \emptyset$, then there exists a partition $(\ha, \hb)$ of $g_{i,a_i}^{*}(\hf)$ such that one of
\[
\hhd(g(\hf)\setminus \hb)\cup \hd(\ha_{(i)}')  \mbox{ and }  \hhd(g(\hf)\setminus \ha)\cup \hd(\hb_{(i)}')
\]
is a star.
\end{cor}
\begin{proof}
Let us choose a partition $\hp_i\cup \hq_i$ of $\binom{[\ell_i-1]}{a_i-1}$ such that $[\ell_i-1]\setminus P \in \hq_i$ for all $P\in \hp_i$ and $[\ell_i-1]\setminus Q \in \hp_i$ for all $Q\in \hq_i$. Let
\[
\ha =\left\{(A_1,\ldots,A_i,\ldots, A_p)\in g_{i,a_i}^*(\hf)\colon A_i\setminus \{\ell_i\} \in \hp_i\right\}
 \]
and
\[
 \hb =\left\{(B_1,\ldots,B_i,\ldots, B_p)\in g_{i,a_i}^*(\hf)\colon B_i\setminus \{\ell_i\} \in \hq_i\right\}.
\]
Clearly, $\ha \cap \hb=\emptyset$.
Without loss of generally assume that $\ha \neq \emptyset$ since $\ha \cup \hb=g_{i,a_i}^{*}(\hf)\neq \emptyset$.
By (iii) of Lemma \ref{lem-key}, we know that $\hb \neq \emptyset$.
Thus, $\ha$ and $\hb$ satisfy the Condition (C1) in Lemma \ref{lem-key4},
and the Conditions (C2)-(C3) by the definitions of $\ha$ and $\hb$.
Hence by Lemma \ref{lem-key}, one of $\hhd(g(\hf)\setminus \hb)\cup \hd(\ha_{(i)}')$  and   $\hhd(g(\hf)\setminus \ha)\cup \hd(\hb_{(i)}')$ is a star.
\end{proof}

\section{Structures of shifted maximum intersecting families}

The following lemma determines all the possible generating sets for the shifted, non-trivial intersecting direct product families with the maximum size.

\begin{lem}\label{lem-4.1}
Let $n_i>2k_i$, $i=1,2,\ldots,p$.
Suppose that $\hf\subset \prod\limits_{1\leq i\leq p}\binom{[n_i]}{k_i}$ is a non-trivial intersecting family of the maximum size.
Let $g(\hf)\in G_*(\mathcal F)$ with  $s^+(g(\hf))=(\ell_1,\ell_2,\ldots,\ell_p)$ such that $\sum_{\vec F\in g(\hf)}|\vec F|$ is minimal.
Then $\ell_t\geq 1$ for some $t\in [p]$ and
  \[
g(\hf)=\left\{\vec{E}=(E_1,E_2,\ldots,E_p)\colon  1\in E_t, \ |\vec{E}|=2, \  |\vec{E}\cap \vec{E}^*|= 1\right\}\cup \{ \vec{E}^*\},
\]
where $E_t^*=[2,\ell_t]$, $E_j^*=[\ell_j]$ for  $j\neq t$.
\end{lem}

Note that if $\ell_t= 1$ then $E_t^*=\emptyset$ and if $\ell_j=0$ then $E_j^*=[0]=\emptyset$.
To gain some intuitive understanding of this lemma, let's consider a simple example.
For instance: if $t=1$, $p=3$ and $\ell_1=\ell_2=\ell_3=2$, then $E_1^*=\{[2,3], [2], [2]\}$ and  $g(\hf)$ is
\[
\{ (\{1,2\},\emptyset,\emptyset), \ (\{1,3\},\emptyset,\emptyset), \ (\{1\},\{1\},\emptyset),  \ (\{1\},\{2\},\emptyset), \ (\{1\},\emptyset,\{1\}), \ (\{1\},\emptyset,\{2\})\}\cup \{E_1^*\}.
\]

We divide the proof into two parts in the following two subsections.
 The first subsection deals with the case of $\ell_i\leq 1$ for all $i=1,2,\ldots,p$,
 while the next one  deals with the case of $\ell_i\ge 2$ for some $i=1,2,\ldots,p$.

\subsection{The case of $\ell_i\le 1$  for all $i=1,2,\ldots,p$}
In this subsection, we deal with the case $\ell_i=0$ or 1 for all $i\in [p]$.
 By switching parts of $\prod_{1\leq i\leq p}\binom{[n_i]}{k_i}$, we may assume that $\ell_1=\ell_2=\cdots=\ell_q=1$ and $\ell_{q+1}=\cdots=\ell_p=0$, and $\frac{n_1}{k_1} \leq \frac{n_2}{k_2} \leq \ldots \leq \frac{n_q}{k_q}$.

Let us define a  generating index set for $\hf$ as follows.
For $\vec{E}\in g(\hf)$,  we have $E_i=\{1\}$ or $\emptyset$ for each $i\in [q]$.
Let $\varphi: g(\hf) \rightarrow 2^{[q]}$ be a mapping with $\varphi(\vec{E}) =\left\{i\in [q] \colon E_i=\{1\}\right\}$.
We can regard $\varphi(\vec{E})$ as the index set of $\vec{E}$.
Denote
\[
\varphi(g(\hf)) =\{\varphi(\vec{E})\colon \vec{E}\in g(\hf)\}.
\]
Clearly, this mapping $\varphi$ is injective.
More importantly, $\varphi(g(\hf))$ contains almost all important information of $g(\hf)$ (so does $\hf$), such as, \\
(P1) $g(\hf)$ is intersecting if and only if $\varphi(g(\hf))$ is intersecting;\\
(P2) $\hf$ is non-trivial if and only if $\varphi(g(\hf))$ is non-trivial.

Let $N= \prod\limits_{1\leq i\leq p} \binom{n_i}{k_i}$. For $I\subset [q]$, define
\[
w_q(I) =\prod_{i\in I}\frac{k_i}{n_i}  \prod_{i\in [q]\setminus I}\left(1-\frac{k_i}{n_i}\right).
\]
Since $\frac{n_1}{k_1} \leq \frac{n_2}{k_2} \leq \ldots \leq \frac{n_q}{k_q}$, we deduce that $w_q(I)\le w_q(I')$ if $I' \preceq I$.
Moreover,
\begin{align*}
|\hf| &= \sum_{\vec{E}\in g(\hf)}\prod_{i\in I(\vec{E})}\binom{n_i-1}{k_i-1} \prod_{i\in [q]\setminus I(\vec{E})}\binom{n_i-1}{k_i}\\[3pt]
&= \sum_{I\in \varphi(g(\hf))}\prod_{i\in I}\binom{n_i-1}{k_i-1} \prod_{i\in [q]\setminus I}\binom{n_i-1}{k_i}\\[3pt]
&=N\sum_{I\in \varphi(g(\hf))}w_q(I).
\end{align*}
Thus by the maximality of $|\hf|$, we can assume that $\varphi(g(\hf))$ is shifted.
\begin{lem}\label{lem-key3}
If $\varphi(g(\hf))$ is shifted, then
\[
\varphi(g(\hf)) =\left\{\{1,j\}\colon 2\leq j\leq q\right\}\cup \{[2,q]\}.
\]
\end{lem}
\begin{proof}
Let us fix a partition $(\hi_1,\hi_2)$ of $2^{[q-1]}$ such that $\{I, [q-1]\setminus I\} \cap \hi_i=1$ for each $I\in 2^{[q-1]}$ and $i=1,2$. Define
\[
\ha=\left\{\vec{A}\in g_q^*(\hf)\colon \varphi(\vec{A})\setminus \{q\}\in \hi_1\right\} \:{\rm and}\: \hb=\left\{\vec{B}\in g_q^*(\hf)\colon \varphi(\vec{B})\setminus \{q\}\in \hi_2\right\}.
\]
Thus $\{\ha,\hb\}$ is a partition of $g_q^*(\hf)$. Recall the definitions of $\ha'_{(q)}$ and $\hb'_{(q)}$, that is, for $\he\in \{\ha,\hb\}$,
\[
\he_{(q)}'=\left\{(E_1,\ldots, E_{q-1},\emptyset, \emptyset,\ldots, \emptyset)\colon (E_1,\ldots, E_{q-1},\{1\},\emptyset, \ldots, \emptyset)\in \he\right\}.
\]
Denote
\[
\hf_1=\hhd(g(\hf)\setminus \hb)\cup \hd(\ha_{(q)}')  \mbox{ and }  \hf_2=\hhd(g(\hf)\setminus \ha)\cup \hd(\hb_{(q)}')
\]

By Lemma 2.4 (iii),  for any $\vec{A}\in g_q^*(\hf)$, there exists exactly a $\vec{B}\in g_q^*(\hf)$ such that $\vec{A}\cap \vec{B}=(\emptyset,\ldots,\emptyset, \{\ell_q\}, \emptyset,\ldots,\emptyset)$, which implies that $\varphi(\vec{A})\cap \varphi(\vec{B})=\{q\}$.
We claim that $\varphi(\vec{A})\cup \varphi(\vec{B})=[q]$, since otherwise suppose that $x\notin \varphi(\vec{A})\cup \varphi(\vec{B})$ for some $x\in [q-1]$,
then there exists some $E \subset \varphi(\vec{A})\setminus \{q\} \cup \{x\}$ such that $E\in \varphi(g(\hf))$ since $\varphi(g(\hf))$ is shifted.
However, $E\cap \varphi(\vec{B})=\emptyset$, contradicting that $\varphi(g(\hf))$ is intersecting.
Thus $\varphi(\vec{A})\cup \varphi(\vec{B})=[q]$,
which equals to $\vec{A}\cup \vec{B}=(\{\ell_1\},\ldots, \{\ell_q\}, \emptyset,\ldots,\emptyset)$.
Hence such $A$ and $B$ occur in pairs, and one of them belongs to $\ha$ and the other to $\hb$.
This implies that $\ha$ and $\hb$ satisfy the three conditions of Lemma \ref{lem-key4}.
It follows that one of $g(\hf)\cup \ha'_{(q)}\setminus \hb$ and $g(\hf)\cup \hb'_{(q)}\setminus \ha$ is a star. By symmetry, we may assume that $g(\hf)\cup \ha'_{(q)}\setminus \hb$ is a star. Since $\varphi(g(\hf))$ is shifted, $g(\hf)\cup \ha'_{(q)}\setminus \hb$ is a star of the first coordinate.
Since $\varphi(g(\hf))$ is non-trivial, $|I|\ge 2$ for all $I\in \varphi(g(\hf))$.
Thus $\{1,q\} \in I$. By shiftedness, $\{1,j\} \in \varphi(g(\hf))$ for all $j=2,3,\ldots,q$. By the maximality of $|\hf|$, we conclude that
\[
\varphi(g(\hf)) =\left\{\{1,j\}\colon 2\leq j\leq q\right\}\cup \{[2,q]\}.
\]
\end{proof}

Suppose that $S_{i,s}(\hj)=\left\{\{1,j\}\colon 2\leq j\leq q\right\}\cup \{[2,q]\}$ for some $1\le i<s\le q$.
By symmetry, we may assume that $(i,s)=(1,2)$ or $(i,s)=(2,3)$.
If $(i,s)=(1,2)$, then we deduce that $\hj=\left\{\{2,j\}\colon j\in \{1,3,\cdots,q\}\right\}\cup \{[\{1,3,\cdots,q\}\}$.
If $(i,s)=(2,3)$, then we deduce that $\hj=\left\{\{1,j\}\colon 2\leq j\leq q\right\}\cup \{[2,q]\}$.
In both cases, we know that $\hj \cong \left\{\{1,j\}\colon 2\leq j\leq q\right\}\cup \{[2,q]\}$, which implies $\varphi(g(\hf)) \cong \left\{\{1,j\}\colon 2\leq j\leq q\right\}\cup \{[2,q]\}$ without the assumption that $\varphi(g(\hf))$ is shifted.
Thus $g(\hf)$ meets the requirements in Lemma \ref{lem-4.1}, and we complete the proof in this case.

\subsection{The case of $\ell_i\ge2$  for some $i=1,2,\ldots,p$}

By symmetry assume that $\ell_1\geq \ldots\geq\ell_r\geq 2$, $\ell_{r+1}=\ldots=\ell_q=1$ and $\ell_{q+1}=\ldots=\ell_p=0$. By Corollaries \ref{lem-key1} and  \ref{lem-key2}, for each $i\in [r]$ there exist $\ha, \hb\subset g^*_i(\hf)$ such that $\ha\cap \hb=\emptyset$ and $g(\hf)\cup  \ha_{(i)}'\setminus \hb$ is a star of the $\sigma(i)$th coordinate for all $i=1,2,\ldots,r$. Since $\hf$ is shifted, we may always assume that the centre is $1$ of the $\sigma(i)$th coordinate.

\begin{claim}\label{claim-3}
If $\sigma(i)\neq i$,  then $1\in E_{\sigma(i)}$ for all $\vec{E}\in g(\hf)\setminus \hb$ and $1\notin B_{\sigma(i)}$ for every $\vec{B}\in \hb$. Moreover, for all $\vec{E}=(E_1,E_2,\ldots,E_p)\in g(\hf)$, we have (i) $E_i=[\ell_i]$ or $E_i=\emptyset$ or $E_i=\{j\}$ for some $j\in [\ell_i]$; (ii) if
$1\notin E_{\sigma(i)}$ then $E_i= [\ell_i]$.
\end{claim}

\begin{proof}
Since  $g(\hf)\cup  \ha_{(i)}'\setminus \hb$ is a star of the $\sigma(i)$th coordinate, we infer that $1\in E_{\sigma(i)}$ for all $(E_1,E_2,\ldots,E_p)\in g(\hf)\setminus \hb$.  Note that $\sigma(i)\neq i$. By Corollaries \ref{lem-key1}, \ref{lem-key2}  and Lemma \ref{lem-key} (iii), we infer that for any $\vec{B}\in \hb$ there exists $\vec{A}\in  \ha$ such that $A_{\sigma(i)}\cap B_{\sigma(i)}=\emptyset$,  implying  $1\notin B_{\sigma(i)}$. Thus $1\notin B_{\sigma(i)}$ for every $\vec{B}\in \hb$.

 Suppose that there is some $(B_1,B_2,\ldots,B_p)\in \hb$ with $B_i\neq [\ell_i]$. Let $x\in [\ell_i]\setminus B_i$. By Lemma \ref{lem-key} (i) there exists $\vec{E}\in g(\hf)$ such that
 \[
 \vec{E}=(E_1,E_2,\ldots,E_p)\preceq (B_1,\ldots,B_i\cup \{x\}\setminus \{\ell_i\},\ldots,B_p).
 \]
Now $\vec{E}\notin \hb$ implies $1\in E_{\sigma(i)}\subset B_{\sigma(i)}$, a contradiction. Thus for all  $(B_1,B_2,\ldots,B_p)\in \hb$, we have $B_i=[\ell_i]$. By Lemma \ref{lem-key} (iii), $A_i=\{\ell_i\}$ for all  $(A_1,A_2,\ldots,A_p)\in \ha$.

Suppose that there exists $\vec{R}=(R_1,R_2,\ldots,R_p)\in g(\hf)\setminus \hb$ with $|R_i|\geq 2$. Then $\vec{R}\notin \ha\cup \hb$. It follows that  $1\in R_{\sigma(i)}$.
Let $\ell_i'=s^+(R_i)$ and let $\vec{R}'=(R_1,\ldots,R_i\setminus\{\ell_i'\},\ldots,R_p)$. Then $\vec{R}'\notin g(\hf)$ since $g(\hf)$ is an anti-chain. It follows that there exists $\vec{B}\in g(\hf)$ such that $|\vec{R}'\cap \vec{B}|=0$. Since $1\in  R_{\sigma(i)}$ implies $1\notin B_{\sigma(i)}$, we infer $\vec{B}\in \hb$. Then $B_i=[\ell_i]$ and $|R_i|\geq 2$ implies $(R_i\setminus \{\ell_i'\})\cap B_i\neq \emptyset$, contradicting $|\vec{R}'\cap \vec{B}|=0$. Thus $|E_i|=0,1$ or $\ell_i$  for all $\vec{E}\in g(\hf)$ and the claim follows.
\end{proof}

We need the following Claim, whose proof is similar to the one above.
\begin{claim}\label{claim-4}
If $\sigma(i)= i$, then  $1\in E_{i}$ for every $\vec{E}\in g(\hf)\setminus \hb$ and $1\notin B_{i}$ for every $\vec{B}\in \hb$. Moreover, for all $\vec{E}=(E_1,E_2,\ldots,E_p)\in g(\hf)$ we have $E_i=[2,\ell_i]$ or $E_i=\{1\}$ or $E_i=\{1,j\}$ for some $j\in [2,\ell_i]$.
\end{claim}

\begin{proof}
Since $g(\hf)\cup  \ha_{(i)}'\setminus \hb$ is a star of the $i$th coordinate, we infer that $1\in E_{i}$ for every $\vec{E}\in g(\hf)\setminus \hb$. By Corollaries \ref{lem-key1}, \ref{lem-key2}  and Lemma \ref{lem-key} (iii), we infer that for any $\vec{B}\in \hb$ there exists $\vec{A}\in  \ha$ such that $A_{i}\cap B_{i}=\{\ell_i\}$,  implying  $1\notin B_{i}$. Thus $1\notin B_{i}$ for every $\vec{B}\in \hb$.

Suppose that there exists some $\vec{B}\in \hb$ with $B_i\subsetneq [2,\ell_i]$.  Then let $x\in [2,\ell_i]\setminus B_i$. By Lemma \ref{lem-key} (i) there exists $\vec{E}\in g(\hf)$ such that
 \[
 \vec{E}=(E_1,E_2,\ldots,E_p)\preceq (B_1,\ldots,B_i\cup \{x\}\setminus \{\ell_i\},\ldots,B_p).
 \]
Since $\vec{E}\in g(\hf)\setminus \hb$, we infer that $1\in E_i\subset B_i\cup \{x\}\setminus \{\ell_i\}$. It follows that $1\in B_i$, a contradiction. Thus $B_i= [2,\ell_i]$ for all $\vec{B}\in \hb$. By Lemma \ref{lem-key} (iii), $A_i=\{1,\ell_i\}$ for all  $\vec{A}\in \ha$.

Recall that $1\in E_i$ for all $\vec{E}\in g(\hf)\setminus \hb$. If $|R_i|\geq 3$ for some $\vec{R}\in  g(\hf)\setminus (\ha\cup \hb)$, let $\ell_i'=s^+(R_i)$ and let $\vec{R}'=(R_1,\ldots,R_i\setminus \{\ell_i'\},\ldots,R_p)$.
Then $\vec{R}'\notin g(\hf)$ since $g(\hf)$ is an anti-chain. It follows that there exists $\vec{B}\in g(\hf)$ such that $|\vec{R}'\cap \vec{B}|=0$. Since $1\in  R_{i}$ implies $1\notin B_{i}$, we infer $\vec{B}\in \hb$. Then $B_i=[2,\ell_i]$ and $|R_i|\geq 3$ implies $(R_i\setminus \{\ell_i'\})\cap B_i\neq \emptyset$, contradicting $|\vec{R}'\cap \vec{B}|=0$.
Therefore, $|R_i|\leq 2$ and $1\in R_i$ for all $\vec{R}\in g(\hf)\setminus (\ha\cup \hb)$. Thus the claim follows.
\end{proof}

\begin{claim}\label{claim-5}
$\sigma(1)=\sigma(2)=\ldots=\sigma(r)$.
\end{claim}

\begin{proof}
By symmetry, we distinguish two cases.

\vspace{5pt}
{\bf Case 1.} $\sigma(1)=1$.

By Claim \ref{claim-4}, one can partition $g(\hf)$ into several classes as follows:
\begin{align*}
&g^{(1)}(\hf)=\{(E_1,E_2,\ldots,E_p)\in g(\hf)\colon E_1=[2,\ell_1]\}, \\[5pt] &g^{(2,j)}(\hf)=\{(E_1,E_2,\ldots,E_p)\in g(\hf)\colon E_1=\{1,j\}\},\  j=2,\ldots,\ell_1,\\[5pt]
&g^{(3)}(\hf)=\{(E_1,E_2,\ldots,E_p)\in g(\hf)\colon E_1=\{1\}\}.
\end{align*}
Since  $(\{1,j\}, \emptyset,\ldots,\emptyset)$ intersects every element of $g(\hf)$ and $\hf$ is a maximum non-trivial intersecting family, it follows that $(\{1,j\},\emptyset,\emptyset,\ldots,\emptyset)\in g(\hf)$. Combining that $g(\hf)$ is an anti-chain, we have
\[
g^{(2,j)}(\hf)=\{(\{1,j\},\emptyset,\emptyset,\ldots,\emptyset)\}, \ j=2,\ldots,\ell_1.
\]
  Now we claim that $\sigma(i)=1$ for all $i=2,3,\ldots,r$. Recall that  $g(\hf)\cup  \ha_{(i)}'\setminus \hb$ is a star of the $\sigma(i)$th coordinate.
  If $\sigma(i)\neq 1$ then by Claim \ref{claim-3} we  have $1\in E_{\sigma(i)}$ for all $\vec{E}\in g(\hf)\setminus \hb$. It  implies that $g^{(2,j)}(\hf)\subset \hb$ for all $j$. However, by $\hb\subset g_{i}^{*}(\hf)$ we have $\ell_i\in B_i$ for all $(B_1,\ldots,B_i,\ldots,B_p)\in \hb$, a contradiction. Thus $\sigma(i)=1$ for all $i$.

\vspace{5pt}
{\bf Case 2.} $\sigma(1)=t$ for some $t\in [2,p]$.

By Claim \ref{claim-3}, one can partition $g(\hf)$  into several classes as follows:
\begin{align*}
&g^{(1)}(\hf)=\{(E_1,E_2,\ldots,E_p)\in g(\hf)\colon E_1=[\ell_1],\ 1\notin E_t\}, \\[5pt]
&g^{(2,j)}(\hf)=\{(E_1,E_2,\ldots,E_p)\in g(\hf)\colon E_1=\{j\},\ 1\in E_t\},\  j=1,\ldots,\ell_1,\\[5pt]
&g^{(3)}(\hf)=\{(E_1,E_2,\ldots,E_p)\in g(\hf)\colon E_1=\emptyset,\ 1\in E_t\}.
\end{align*}

Similar to Case 1, we have
\[
g^{(2,j)}(\hf)=\{(\{j\},\emptyset,\ldots,\emptyset,\{1\},\emptyset,\ldots,\emptyset)\}, \ j=1,2,\ldots,\ell_1.
\]
Recall that  $g(\hf)\cup  \ha_{(i)}'\setminus \hb$ is a star of the $\sigma(i)$th coordinate.  Suppose indirectly that $\sigma(i)\neq  t$ for some $2\leq i\leq r$.   By Claims \ref{claim-3} and \ref{claim-4} we  have $1\in E_{\sigma(i)}$ for all $\vec{E}\in g(\hf)\setminus \hb$. It follows that $(\{j\},\emptyset,\ldots,\emptyset,\{1\},\emptyset,\ldots,\emptyset)\in \hb$. Since $\hb\subset g_{i}^*(\hf)$, $\ell_i\in B_i$ for all $(B_1,B_2,\ldots,B_p)\in \hb$, contradicting $\ell_i\geq 2$ for $i\leq r$. Thus $\sigma(i)=t$ for all $i=1,2,\ldots,r$.
\end{proof}

By Claim \ref{claim-5}, we may assume that $\sigma(1)=\sigma(2)=\ldots=\sigma(r)=t$ for some $t\in [p]$.

\begin{claim}\label{claim-8}
Let $\vec{E}^*=(E_1^*,E_2^*,\ldots,E_p^*)\in g(\hf)$ with $1\notin E_t^*$. Then $E_i^* =[\ell_i]$ for $i\neq t$ and $E_t^* =[2,\ell_t]$.
\end{claim}
\begin{proof}
Fix an arbitrary set $\vec{E}=(E_1,E_2,\ldots,E_p)\in g(\hf)$ with $1\notin E_t$. Since $\ell_i=0$ for $q+1\leq i\leq p$, we infer  $E_i=\emptyset=[0]$. Thus we only need to consider the first $q$ coordinates. We distinguish two cases.

{\bf Case 1.} $1\leq t\leq r$.

By permuting parts of $\prod_{1\leq i\leq p}\binom{[n_i]}{k_i}$, we may assume $t=1$. By Claim  \ref{claim-3}, we infer that $E_i =[\ell_i]$ for $2\leq i\leq r$. By Claim \ref{claim-4} we have $E_1 =[2,\ell_1]$. Thus, for any $\vec{E}\in g(\hf)$ with $1\notin E_1$ we have
 \begin{align}\label{ineq-4.1}
 (E_1,E_2\ldots,E_r)=([2,\ell_1],[\ell_2]\ldots,[\ell_r]).
 \end{align}
 It follows that for every $x\in [\ell_i]$ with $2\leq i\leq r$ and for every $y\in [2,\ell_1]$,
 \begin{align}\label{ineq-4.12}
  (\{1\},\emptyset,\ldots,\emptyset,\{x\},\emptyset,\ldots,\emptyset)\in g(\hf)\ {\rm and}\  (\{1,y\},\emptyset,\ldots,\emptyset)\in g(\hf).
 \end{align}

Let us fix a $j\in [r+1,q]$. Define
\[
\hc=\{\vec{E}\in g(\hf)\colon E_j=\{1\},\ (E_1,E_2\ldots,E_r)=(\{1\},\emptyset,\ldots,\emptyset)\},
\]
\[
\hd=\{\vec{E}\in g(\hf)\colon E_j=\{1\},\ (E_1,E_2\ldots,E_r)=([2,\ell_1],[\ell_2]\ldots,[\ell_r])\},
\]
 \[
\hc'= \{(E_1,\ldots,E_j\setminus\{1\},\ldots,E_p)\colon \vec{E}\in \hc\},\ \hd'= \{(E_1,\ldots,E_j\setminus\{1\},\ldots,E_p)\colon \vec{E}\in \hd\}.
 \]

Since $\ell_j=1$, there exists $\vec{E}\in g(\hf)$ with $E_j=\{1\}$. If $1\notin E_{1}$, then by \eqref{ineq-4.1} we have $(E_1,E_2\ldots,E_r)=([2,\ell_1],[\ell_2]\ldots,[\ell_r])$. It follows that $\vec{E}\in \hd$. If $1\in E_1$, then $(E_1,E_2\ldots,E_r)=(\{1\},\emptyset,\ldots,\emptyset)$.
Since otherwise $ (\{1\},\emptyset,\ldots,\emptyset,\{x\},\emptyset,\ldots,\emptyset) \preceq \vec{E}$ or $ (\{1,y\},\emptyset,\ldots,\emptyset) \preceq \vec{E}$, however,
$  (\{1\},\emptyset,\ldots,\emptyset,\{x\},\emptyset,\ldots,\emptyset)\in g(\hf)$ and $ (\{1,y\},\emptyset,\ldots,\emptyset)\in g(\hf)$ by \eqref{ineq-4.12} , which contradicts that $g(\hf)$ is an anti-chain. It follows that $\vec{E}\in \hc$.
Combining Lemma \ref{lem-key} (iii), we infer that $\hc$ and $\hd$ satisfy the three conditions of Lemma \ref{lem-key4}.
Then by Lemma \ref{lem-key4}, one of $g(\hf)\cup \hc'\setminus \hd$ and $g(\hf)\cup \hd'\setminus \hc$ is a star.
By \eqref{ineq-4.12},  we infer that $g(\hf)\cup \hc'\setminus \hd$ is  a star of the first coordinate with centre 1.  If there exists an $\vec{E}\in g(\hf)$ with $1\notin E_1$ such that $E_j=\emptyset$, then $\vec{E}\in (g(\hf)\setminus \hd)\cup \hc'$, a contradiction. Thus, for any $\vec{E}\in g(\hf)$ with $1\notin E_1$ we have $E_j=\{1\}$ for all $j=r+1,\ldots,q$.

{\bf Case 2.} $r+1\leq t\leq q$.

The proof is similar to above.
By permuting parts of $\prod_{1\leq i\leq p}\binom{[n_i]}{k_i}$, we may assume $t=r+1$.  Let $\vec{E}\in g(\hf)$ with $1\notin E_{r+1}$.  By Claim  \ref{claim-3}, we infer that $E_i =[\ell_i]$ for $1\leq i\leq r$. Since $1\notin E_{r+1}$ and $\ell_{r+1}=1$ we have $E_{r+1} =\emptyset$. Thus,
for any $\vec{E}\in g(\hf)$ with $1\notin E_{r+1}$ we have
 \begin{align}\label{ineq-4.2}
 (E_1,E_2\ldots,E_r,E_{r+1})=([\ell_1],[\ell_2]\ldots,[\ell_r],\emptyset).
 \end{align}
 It follows that for every $x\in [\ell_i]$ with $i\leq r$,
 \begin{align}\label{ineq-4.22}
(\emptyset,\ldots,\{x\},\ldots,\emptyset,\{1\},\ldots,\emptyset)\in g(\hf).
\end{align}

Let us fix a $j\in [r+2,q]$. Define
\[
\hg=\{\vec{E}\in g(\hf)\colon E_j=\{1\},\ (E_1,E_2\ldots,E_r,E_{r+1})=([\ell_1],[\ell_2]\ldots,[\ell_r],\emptyset)\},
\]
\[
\hh=\{\vec{E}\in g(\hf)\colon E_j=\{1\},\ (E_1,E_2,\ldots,E_r,E_{r+1})=(\emptyset,\emptyset,\ldots,\emptyset,\{1\})\},
\]
 \[
\hg'= \{(E_1,\ldots,E_j\setminus\{1\},\ldots,E_p)\colon \vec{E}\in \hg\},\ \hh'= \{(E_1,\ldots,E_j\setminus\{1\},\ldots,E_p)\colon \vec{E}\in \hh\}.
 \]
 Since $\ell_j=1$, there exists $\vec{E}\in g(\hf)$ with $E_j=\{1\}$. If $1\notin E_{r+1}$, then by \eqref{ineq-4.2} we have $(E_1,E_2\ldots,E_r,E_{r+1})=([\ell_1],[\ell_2]\ldots,[\ell_r],\emptyset)$.
 It follows that $\vec{E}\in \hg$.
 If $1\in E_{r+1}$ then by \eqref{ineq-4.22} and the fact that $g(\hf)$ is an anti-chain we infer that $(E_1,E_2,\ldots,E_r,E_{r+1})=(\emptyset,\emptyset,\ldots,\emptyset,\{1\})$.
 It follows that $\vec{E}\in \hh$.
Combining Lemma \ref{lem-key} (iii), we infer that $\hg$ and $\hh$ satisfy the three conditions of Lemma \ref{lem-key4}.
Then by Lemma \ref{lem-key4}, one of $g(\hf)\cup \hg'\setminus \hh$ and $g(\hf)\cup \hh'\setminus \hg$ is a star.
 By \eqref{ineq-4.22}, we infer that $(g(\hf)\setminus \hh)\cup \hg'$ is  a star of the $(r+1)$th coordinate with centre 1.  If there exists an $\vec{E}\in g(\hf)$ with $1\notin E_{r+1}$ such that $E_j=\emptyset$, then $\vec{E}\in (g(\hf)\setminus \hh)\cup \hg'$, a contradiction. Thus, for any $\vec{E}\in g(\hf)$ with $1\notin E_{r+1}$ we have $E_j=\{1\}$ for all $j=r+2,\ldots,q$. Therefore the claim holds.
\end{proof}

By the maximality of $|\hf|$ and Claim \ref{claim-8}, we conclude that
 \[
g(\hf)=\left\{\vec{E}=(E_1,E_2,\ldots,E_p)\colon  1\in E_t, \ |\vec{E}|=2, \  |\vec{E}\cap \vec{E}^*|= 1\right\}\cup \{ \vec{E}^*\},
\]
with $E_t^*=[2,\ell_t]$ and $E_j^*=[\ell_j]$ for  $j\neq t$. Thus the lemma is proven.

\section{Proof of Theorem \ref{thm-1}}

\begin{proof}
[Proof of the inequality]
Let $\hf_0\subset \prod\limits_{1\leq i\leq p}\binom{[n_i]}{k_i}$ be a  non-trivial intersecting family of the maximum size.
Using Lemma \ref{thm-key}, one can obtain a shifted, non-trivial intersecting family $\hf$ with $|\hf|=|\hf_0|$ by repeated shifting.
Let $g(\hf)\in G_*(\mathcal F)$ with  $s^+(g(\hf))=(\ell_1,\ell_2,\ldots,\ell_p)$ such that $\sum_{\vec F\in g(\hf)}|\vec F|$ is minimal.
By Lemma \ref{lem-4.1}, we have  $\ell_t\geq 1$ for some $t\in [p]$ and
 \[
g(\hf)=\left\{\vec{E}=(E_1,E_2,\ldots,E_p)\colon  1\in E_t, \ |\vec{E}|=2, \  |\vec{E}\cap \vec{E}^*|= 1\right\}\cup \{ \vec{E}^*\},
\]
where $E_t^*=[2,\ell_t]$, $E_j^*=[\ell_j]$ for  $j\neq t$.
Clearly we have $1\leq \ell_t\leq k_t+1$ and $0\leq \ell_j\leq k_j$ for $j\neq t$.

\begin{claim}\label{claim-4.1}
For each $i\in [p]\setminus \{t\}$, we have $\ell_i=0$ or $k_i$.
\end{claim}
\begin{proof}
Suppose for contradiction that $1\leq \ell_i\leq k_i-1$. Define  $\vec{R}=(R_1,R_2,\ldots,R_p)$, $\vec{R}'=(R_1',R_2',\ldots,R_p')$ and $\vec{R}''=(R_1'',R_2'',\ldots,R_p'')$ by letting $R_t=R_t'=R_t''=\{1\}$, $R_j=R_j'=R_j''=\emptyset$ for all $j\in [p]\setminus \{t,i\}$ and $R_i=\{\ell_i\}$, $R_i'=\{\ell_{i}+1\}$, $R_i''=\{\ell_i,\ell_{i}+1\}$. Define  $\vec{E}'=(E_1',E_2',\ldots,E_p')$ and $\vec{E}''=(E_1'',E_2'',\ldots,E_p'')$ by letting $E_t'=E_t''=[2,\ell_t]$, $E_j'=E_j''=[\ell_j]$ for all $j\in [p]\setminus \{t,i\}$ and $E_i'=[\ell_i+1]$, $E_i''=[\ell_i-1]\cup \{\ell_i+1\}$.
Define
\[
\hf_1=\hhd(g(\hf)\cup \{\vec{R}',\vec{E}'\}\setminus \{\vec{E}^*\})\mbox{ and }\hf_2=\hhd(g(\hf)\cup \{\vec{R}'',\vec{E}''\}\setminus \{\vec{R}\}).
\]
It is straightforward to verify that both  $\hf_1$ and $\hf_2$ are both non-trivial intersecting. Note that
\begin{small}
\[
|\hhd(\vec{R}')\setminus \hf| = \binom{n_i-\ell_i-1}{k_i-1}\binom{n_t-\ell_t}{k_t-1}\prod_{j\in [p]\setminus \{t,i\}}\binom{n_j-\ell_j}{k_j} =:w_1\binom{n_i-\ell_i-1}{k_i-1}
\]\end{small}
and
\begin{small}
\[
|\hhd(\vec{E}^*)\setminus \hhd(g(\hf)\cup \{\vec{E}'\})| =\binom{n_i-\ell_i-1}{k_i-\ell_i}\binom{n_t-\ell_t}{k_t-\ell_t+1}\prod_{j\in [p]\setminus \{t,i\}}\binom{n_j-\ell_j}{k_j-\ell_j}=:w_2\binom{n_i-\ell_i-1}{k_i-\ell_i}.
\]\end{small}
It follows that
\begin{align*}
|\hf_1| &= |\hf|+|\hhd(\vec{R}')\setminus \hf|-|\hhd(\vec{E}^*)\setminus \hhd(g(\hf)\cup \{\vec{E}'\})|\\[3pt]
&= |\hf|+w_1\binom{n_i-\ell_i-1}{k_i-1}-w_2\binom{n_i-\ell_i-1}{k_i-\ell_i}.
\end{align*}
Since
\begin{small}
\[
|\hhd(\vec{E}'')\setminus \hf| =\binom{n_i-\ell_i-1}{k_i-\ell_i}\binom{n_t-\ell_t}{k_t-\ell_t+1}\prod_{j\in [p]\setminus \{t,i\}}\binom{n_j-\ell_j}{k_j-\ell_j}=w_2\binom{n_i-\ell_i-1}{k_i-\ell_i}
\]
\end{small}
and
\begin{small}
\[
|\hhd(\vec{R})\setminus \hhd(g(\hf)\cup \{\vec{R}''\})| =\binom{n_i-\ell_i-1}{k_i-1}\binom{n_t-\ell_t}{k_t-1}\prod_{j\in [p]\setminus \{t,i\}}\binom{n_j-\ell_j}{k_j} =w_1\binom{n_i-\ell_i-1}{k_i-1},
\]
\end{small}
it follows that
\[
|\hf_2| = |\hf|+w_2\binom{n_i-\ell_i-1}{k_i-\ell_i}-w_1\binom{n_i-\ell_i-1}{k_i-1}.
\]
By the maximality of  $|\hf|$, we infer that
\[
w_2\binom{n_i-\ell_i-1}{k_i-\ell_i}=w_1\binom{n_i-\ell_i-1}{k_i-1}.
\]
Then  $\hf_2$ is also a  non-trivial intersecting family of the maximum size. Moreover, $\hf_2$ is shifted with $g(\hf_2)=g(\hf)\cup \{\vec{R}'',\vec{E}''\}\setminus \{\vec{R}\}$, contradicting Lemma \ref{lem-4.1}. Thus the claim holds.
\end{proof}

\begin{claim}\label{claim-4.2}
If $2\leq \ell_t\leq k_t$, then either $\hf=\hht(\vec{\mathstrut n},\vec{\mathstrut k},t)$ with $k_t=3$ or $\hf=\hh\hm_2(\vec{\mathstrut n},\vec{\mathstrut k},\{t\},i)$ with $k_t=2$ and $k_i=1$.
\end{claim}
\begin{proof}
Part of this proof is similar to the one above.
Define  $\vec{R}=(R_1,R_2,\ldots,R_p)$, $\vec{R}'=(R_1',R_2',\ldots,R_p')$ and $\vec{R}''=(R_1'',R_2'',\ldots,R_p'')$ by letting  $R_j=R_j'=R_j''=\emptyset$ for all $j\in [p]\setminus \{t\}$ and $R_t=\{1,\ell_t\}$, $R_t'=\{1,\ell_t+1\}$, $R_t''=\{1,\ell_t,\ell_t+1\}$. Define  $\vec{E}'=(E_1',E_2',\ldots,E_p')$ and $\vec{E}''=(E_1'',E_2'',\ldots,E_p'')$ by letting $E_j'=E_j''=[\ell_j]$ for all $j\in [p]\setminus \{t\}$ and $E_t'=[2,\ell_t+1]$, $E_t''=[2,\ell_t-1]\cup \{\ell_t+1\}$.
Define
\[
\hf_1=\hhd(g(\hf)\cup \{\vec{R}',\vec{E}'\}\setminus \{\vec{E}^*\})\mbox{ and }\hf_2=\hhd(g(\hf)\cup \{\vec{R}'',\vec{E}''\}\setminus \{\vec{R}\}).
\]
Then it is easy to check that  $\hf_1$ and $\hf_2$ are both non-trivial intersecting. Note that
\[
|\hhd(\vec{R}')\setminus \hf| = \binom{n_t-\ell_t-1}{k_t-2}\prod_{j\in [p]\setminus \{t\}}\binom{n_j-\ell_j}{k_j} =:w_1\binom{n_t-\ell_t-1}{k_t-2}
\]
and
\[
|\hhd(\vec{E}^*)\setminus \hhd(g(\hf)\cup \{\vec{E}'\})| =\binom{n_t-\ell_t-1}{k_t-\ell_t+1}\prod_{j\in [p]\setminus \{t\}}\binom{n_j-\ell_j}{k_j-\ell_j}=:w_2\binom{n_t-\ell_t-1}{k_t-\ell_t+1}.
\]
It follows that
\begin{align*}
|\hf_1| &= |\hf|+|\hhd(\vec{R}')\setminus \hf|-|\hhd(\vec{E}^*)\setminus \hhd(g(\hf)\cup \{\vec{E}'\})|\\[3pt]
&= |\hf|+w_1\binom{n_t-\ell_t-1}{k_t-2}-w_2\binom{n_t-\ell_t-1}{k_t-\ell_t+1}.
\end{align*}
Since
\[
|\hhd(\vec{E}'')\setminus \hf| =\binom{n_t-\ell_t-1}{k_t-\ell_t+1}\prod_{j\in [p]\setminus \{t\}}\binom{n_j-\ell_j}{k_j-\ell_j}=w_2\binom{n_t-\ell_t-1}{k_t-\ell_t+1}
\]
and
\[
|\hhd(\vec{R})\setminus \hhd(g(\hf)\cup \{\vec{R}''\})| =\binom{n_t-\ell_t-1}{k_t-2}\prod_{j\in [p]\setminus \{t\}}\binom{n_j-\ell_j}{k_j} =w_1\binom{n_t-\ell_t-1}{k_t-2},
\]
it follows that
\[
|\hf_2| = |\hf|+w_2\binom{n_t-\ell_t-1}{k_t-\ell_t+1}-w_1\binom{n_t-\ell_t-1}{k_t-2}.
\]
By the maximality of  $|\hf|$, we infer that
\begin{align}\label{ineq-4.23}
w_2\binom{n_t-\ell_t-1}{k_t-\ell_t+1}=w_1\binom{n_t-\ell_t-1}{k_t-2}.
\end{align}
Then  $\hf_1$ and $\hf_2$ are both non-trivial intersecting families of the maximum size.
Note that $\hf_1$ is shifted, and $\hf_2$ is shifted if and only if $\ell_t \ge 4$.
Hence for $\ell_t \ge 4$, we obtain a $\hf_2$ with different structure, contradicting Lemma \ref{lem-4.1}. So assume that  $\ell_t\le 3$.

By Claim \ref{claim-4.1}, we have $\ell_i=0$ or $k_i$ for all $i\in [p]\setminus \{t\}$. It follows that
\[
w_2=\prod_{j\in [p]\setminus \{t\}}\binom{n_j-\ell_j}{k_j-\ell_j} \leq \prod_{j\in [p]\setminus \{t\}}\binom{n_j-\ell_j}{k_j}=w_1,
\]
where the equality holds if and only if  $\ell_j=0$ for all $j\in [p]\setminus \{t\}$.

Firstly, assume that $\ell_t=3$. By \eqref{ineq-4.23}, we have $\sum_{i\in [p]\setminus \{t\}} \ell_i= 0$.
If $k_t \geq 4$ then let
\[
\hf_3=\left\{\vec{F}\in \prod_{1\leq i\leq p} \binom{[n_i]}{k_i}\colon F_t=[2,k_t+1] \mbox{ or } 1\in F_t, F_t\cap[2,k_t+1]\neq \emptyset\right\}.
\]
By Theorem \ref{thm-hm} and $k_t\geq 4$, we infer $|\hf|=|\hf_1|<|\hf_3|$, contradicting the maximality of $|\hf|$. Thus we have $k_t=\ell_t=3$  and $\hf=\hht(\vec{\mathstrut n},\vec{\mathstrut k},t)$ with $k_t=3$.

Now assume that $\ell_t=2$. By \eqref{ineq-4.23}, we have $\sum_{i\in [p]\setminus \{t\}} \ell_i> 0$.
If $k_t\geq 3$ then let
\[
\hg_1=\left\{\vec{E}=(E_1,E_2,\ldots,E_p)\colon 1\in E_t,\ |\vec{E}\cap \vec{G}^*|= 1\right\}\cup \{ \vec{G}^*\},
\]
where $G_t^*=[2,\ell_t+2]$ and $G_j^*=[\ell_j]$ for  $j\neq t$.
\begin{align*}
|\hhd(\hg_1)| &=|\hf_1|+w_1\binom{n_t-\ell_t-2}{k_t-2}-w_2\binom{n_t-\ell_t-2}{k_t-\ell_t}\\[3pt]
&=|\hf_1|+w_1\binom{n_t-4}{k_t-2}-w_2\binom{n_t-4}{k_t-2}>|\hf_1|=|\hf|,
\end{align*}
contradicting the maximality of $|\hf|$.  Thus $k_t=\ell_t=2$. If $\sum_{j\in [p]\setminus \{t\}} \ell_j\geq 2$, then let
\[
\hg_2=\left\{\vec{E}=(E_1,E_2,\ldots,E_p)\colon 1\in E_t,\ |\vec{E}\cap \vec{G}'|= 1\right\}\cup \{ \vec{G}'\}
\]
with $G_t'=\emptyset$ and $G_j'=[\ell_j]$ for  $j\neq t$. By \eqref{ineq-4.23} and $\ell_t=k_t=2$, $w_1=w_2(n_t-3)$.
Then by $n_t>2\ell_t=4$,
\begin{align*}
|\hhd(\hg_2)| &=|\hf|+w_2\binom{n_t-\ell_t}{k_t}-w_1\binom{n_t-\ell_t}{k_t-\ell_t}=|\hf|+w_2\binom{n_t-2}{2}-w_2(n_t-3)>0,
\end{align*}
a contradiction. Thus $\sum_{j\in [p]\setminus \{t\}} \ell_j=1$. By Claim \ref{claim-4.1} there exists $i\in [p]\setminus \{t\}$ such that $\ell_i=k_i=1$ and $\ell_j=0$ for all $j\in [p]\setminus \{t,i\}$. In this case we have $\hf=\hh\hm_2(\vec{\mathstrut n},\vec{\mathstrut k},\{t\},i)$ with $k_t=2$ and $k_i=1$.
\end{proof}

\medskip
Let $S=\{j\in [p]\setminus \{t\}\colon \ell_j=k_j\}$.
If $\ell_t=1$, then $\hf=\hh\hm_2(\vec{\mathstrut n},\vec{\mathstrut k},S,t)$.  If $\ell_t=k_t+1$, then $\hf=\hh\hm_1(\vec{\mathstrut n},\vec{\mathstrut k},S,t)$. If $2\leq \ell_t\leq k_t$, then by Claim \ref{claim-4.2} we have $\hf=\hht(\vec{\mathstrut n},\vec{\mathstrut k},t)$ with $k_t=3$ or $\hf=\hh\hm_2(\vec{\mathstrut n},\vec{\mathstrut k},\{t\},i)$ with $k_t=2$ and $k_i=1$. Since $|\hht(\vec{\mathstrut n},\vec{\mathstrut k},t)|=|\hh\hm_1(\vec{\mathstrut n},\vec{\mathstrut k},\emptyset,t)|$ for $k_t=3$, we conclude that
\[
|{\hf}| = \max \{M_1(\vec{\mathstrut n},\vec{\mathstrut k}),M_2(\vec{\mathstrut n},\vec{\mathstrut k})\}.
\]

{\em Proof of the extremal structures.}
Before the proof, we introduce a definition we need.
Two families of sets $\mathcal{A}$ and $\mathcal{B}$ are called cross-intersecting if $A\cap B\neq \emptyset$ for all $A\in \mathcal{A}$,  $B\in \mathcal{B}$.
Recall that $\hf$ is obtained from $\hf_0$ by a series of shifting operations.
We will show that $\hf$ is isomorphic to $\hf_0$ in the cases of $\hf=\hht(\vec{\mathstrut n},\vec{\mathstrut k},t)$, or $\hf=\hh\hm_1(\vec{\mathstrut n},\vec{\mathstrut k},S,t)$ for some $(S,t)\in \Omega_1(\vec{\mathstrut n},\vec{\mathstrut k})$, or $\hf=\hh\hm_2(\vec{\mathstrut n},\vec{\mathstrut k},S,t)$ for some $(S,t)\in \Omega_2(\vec{\mathstrut n},\vec{\mathstrut k})$.
For the contrary suppose that there exist a family $\hf'$ such that $\hf'$ is not isomorphic to $\hf$, and $S_{i,j}^{\ell}(\hf')=\hf$ for some integers $i,j,\ell$ with $1\le \ell \le p$ and $1\le i<j\le n_{\ell}$.

{\bf Case 1}. $\hf=\hht(\vec{\mathstrut n},\vec{\mathstrut k},t)$.

In this case, it must be $\ell=t$, $i\in [3]$ and $j\in [4,n_t]$.
Without loss of generality assume that $i=1$ and $j=4$.
Denote $\hf_t'=\{F\cap [n_t]:F\in \hf'\}$.
Let
\[\hh_1=\{H\in \binom{[5,n_t]}{k_t-2}:H\cup \{1,2\}\in \hf_t'\} \text{ and } \hh_2=\{H\in \binom{[5,n_t]}{k_t-2}:H\cup \{2,4\}\in \hf_t'\}.\]
Since $S_{1,4}^{t}(\hf')=\hf$, we have $\hh_1\cup \hh_2=\binom{[5,n_t]}{k_t-2}$ and $\hh_1\cap \hh_2=\emptyset$.
Furthermore, since $\hf'$ is not isomorphic to $\hf$, we have $\hh_2\neq \emptyset$.
Let
\[\hh_3=\{H\in \binom{[5,n_t]}{k_t-2}:H\cup \{1,3\}\in \hf_t'\}.\]
Notice that $\hh_3=\binom{[5,n_t]}{k_t-2}$.
Since $\hf'$ is intersecting, $\hh_2$ and $\hh_3$ are non-empty cross intersecting.
Let $A\in \hh_2$. Since $n_t>2k_t$, we can choose $B\in \binom{[5,n_t]\setminus A}{k_t-2}$.
Thus $B\in \hh_3$, and $A\cap B=\emptyset$, a contradiction.

{\bf Case 2}. $\hf=\hh\hm_1(\vec{\mathstrut n},\vec{\mathstrut k},S,t)$ for some $(S,t)\in \Omega_1(\vec{\mathstrut n},\vec{\mathstrut k})$.

Recall the definition of $\hh\hm_1(\vec{\mathstrut n},\vec{\mathstrut k},S,t)=\hh\hm_{11}\cup \hh\hm_{12}$, where
\begin{small}\[\hh\hm_{11} =\left\{\vec{F}\in\prod_{1\leq i\leq p} \binom{[n_i]}{k_i}\colon  F_j=[k_j] \mbox{ for all }j\in S\mbox{ and }F_t=[2,k_{t}+1]\right\}\]\end{small}
and
\begin{small}
\[\hh\hm_{12}=\left\{\vec{F}\in\prod_{1\leq i\leq p} \binom{[n_i]}{k_i}\colon 1\in F_t,\ F_t\cap [2,k_t+1]\neq \emptyset \mbox{ or } F_j\cap [k_j]\neq \emptyset\mbox{ for some }j\in S\right\}.\]
\end{small}
It is easy to see that $S_{i,j}^{\ell}(\hh\hm_{11})$ is isomorphic to $\hh\hm_{11}$.
So we consider $\hh\hm_{12}$.
Suppose that $S_{i,j}^{\ell}(\hh\hm_{12}')=\hh\hm_{12}$ for some integers $i,j,\ell$ with $1\le \ell \le p$ and $1\le i<j\le n_{\ell}$ such that $\hh\hm_{12}'$ is not isomorphic to $\hh\hm_{12}$.
Clearly, $\ell \notin [p]\setminus (S\cup \{t\})$.

Firstly, suppose that $\ell=t$.
It must be $i=1$. It is not hard to see that $\hh\hm_{12}'\cup \{[2,k_t+1]\}$ is non-trivial intersecting.
 Combining with the fact that $\hh\hm_{12}'\cup \{[2,k_t+1]\}$ is not isomorphic to $\hh\hm(n,k)$ and Theorem \ref{thm-hm}, we have $|\hh\hm_{12}'\cup \{[2,k_t+1]\}|<|\hh\hm_{12}|-1$, i.e., $|\hh\hm_{12}'|<|\hh\hm_{12}|$, contradicting $|\hh\hm_{12}'|=|\hh\hm_{12}|$.

Now assume that $\ell\in S\setminus \{t\}$.
Then it must be $i\in [k_{\ell}]$ and $j\in [k_{\ell}+1, n_{\ell}]$.
Since $S_{i,j}^{\ell}(\hf')=\hf$, there exists $D\in \hf'_{\ell}$ with $j\in D$ such that $D\cap [k_{\ell}]=\emptyset$.
Let $F\in \hf'\setminus \hh\hm_{12}$ with $F_{\ell}=D$, $F_t\cap [2,k_t+1]=\emptyset$, and $F_s\cap [k_s]=\emptyset$ for every $s\in S\setminus \{t,\ell\}$.
For every $q\in [p]\setminus S$, since $n_q>2k_q$, we can choose $G\in \hh\hm_{11}$ with $G_q\cap F_q=\emptyset$.
However, the fact $F\cap G=\emptyset$ contradicts to that $\hf'$ is intersecting.

{\bf Case 3}. $\hf=\hh\hm_2(\vec{\mathstrut n},\vec{\mathstrut k},S,t)$ for some $(S,t)\in \Omega_2(\vec{\mathstrut n},\vec{\mathstrut k})$.

The subcase of $\ell \neq t$ is similar to Case 2.
So we only consider $\ell=t$. However, we can get a contradiction from this subcase by Theorem \ref{EKR} directly.
Thus we complete the proof.
\end{proof}

\section*{Acknowledgements}
This paper originated in 2023, when the first two authors visited the third author at Shaoxing, and the first draft was completed in 2023. 
The first author is supported by the National Natural Science Foundation of China  (No. 12471316).
The last author is supported by the National Natural Science Foundation of China  (No.12371332 and No.11971439).

\section*{AI-use declaration}
No Generative AI tools were used at any stage.

\end{document}